\documentclass[12pt,reqno]{amsart}
  \usepackage{amssymb}
  \usepackage{a4}
  \usepackage{a4wide}
  \usepackage{amssymb}
  \usepackage{amsmath}
  \usepackage{amsthm}
  \usepackage{amsfonts}
  \usepackage{amscd}
  \usepackage{xspace}
  \usepackage{boxedminipage}
  \usepackage{nomencl}

  \usepackage{booktabs}
  \usepackage{multirow}
  \usepackage[figuresright]{rotating}
  \usepackage[format=hang,font=small]{caption} 
  \usepackage[curve,graph]{xypic}
  \usepackage{lscape}
  \usepackage{subfigure}

  \usepackage{stmaryrd}

\makeatletter
  \CheckCommand*\refstepcounter[1]{\stepcounter{#1}%
      \protected@edef\@currentlabel
       {\csname p@#1\endcsname\csname the#1\endcsname}%
  }
  \renewcommand*\refstepcounter[1]{\stepcounter{#1}%
    \protected@edef\@currentlabel
      {\csname p@#1\expandafter\endcsname\csname the#1\endcsname}%
  }
  \def\labelformat#1{\expandafter\def\csname p@#1\endcsname##1}
  \DeclareRobustCommand\Ref[1]{\protected@edef\@tempa{\ref{#1}}%
     \expandafter\MakeUppercase\@tempa
  }
  \makeatother

  \makeatletter
  \newcommand{\numberlike}[2]{%
     \expandafter\def\csname c@#1\endcsname{%
         \expandafter\csname c@#2\endcsname}%
  }
  \makeatother

  \def\DefaultNumberTheoremWithin{section}
  \labelformat{section}{Section~#1}
\labelformat{subsection}{Section~#1}
 \theoremstyle{plain}
  \newtheorem{lemma}{Lemma}
     \numberwithin{lemma}{\DefaultNumberTheoremWithin}
     \labelformat{lemma}{Lemma~#1}
  
     \numberwithin{claim}{\DefaultNumberTheoremWithin}
     \numberlike{claim}{lemma}
     \labelformat{claim}{Claim~#1}
  \newtheorem{theorem}{Theorem}
     \numberwithin{theorem}{\DefaultNumberTheoremWithin}
     \numberlike{theorem}{lemma}
     \labelformat{theorem}{Theorem~#1}
  \newtheorem{corollary}{Corollary}
     \numberwithin{corollary}{\DefaultNumberTheoremWithin}
     \numberlike{corollary}{lemma}
     \labelformat{corollary}{Corollary~#1}
  \newtheorem{proposition}{Proposition}
     \numberwithin{proposition}{\DefaultNumberTheoremWithin}
     \numberlike{proposition}{lemma}
     \labelformat{proposition}{Proposition~#1}
  
     \numberwithin{conjecture}{\DefaultNumberTheoremWithin}
     \numberlike{conjecture}{lemma}
     \labelformat{conjecture}{Conjecture~#1}

  \theoremstyle{definition}
  \newtheorem{definition}{Definition}
     \numberwithin{definition}{\DefaultNumberTheoremWithin}
     \numberlike{definition}{lemma}
     \labelformat{definition}{Definition~#1}

  \theoremstyle{definition}
  \newtheorem{question}{Question}
     \numberwithin{question}{\DefaultNumberTheoremWithin}
     \numberlike{question}{lemma}
     \labelformat{question}{Question~#1}

  \theoremstyle{definition}
  
     \numberwithin{problem}{\DefaultNumberTheoremWithin}
     \numberlike{problem}{lemma}
     \labelformat{problem}{Problem~#1}

  \theoremstyle{remark}
  \newtheorem{remark}{Remark}
     \numberwithin{remark}{\DefaultNumberTheoremWithin}
     \numberlike{remark}{lemma}
     \labelformat{remark}{Remark~#1}
  \theoremstyle{remark}

  \newtheorem{example}{Example}
     \numberwithin{example}{\DefaultNumberTheoremWithin}
     \numberlike{example}{lemma}
     \labelformat{example}{Example~#1}
  
     \labelformat{case}{Case~#1}
     \numberwithin{case}{lemma}
  
     \labelformat{step}{Step~#1}
     \numberwithin{step}{lemma}

  \newcommand{\Hilb}{\mathrm{Hilb}}

   \newcommand{\last}{\mathrm{last}}
   
  \newcommand{\supp}{\mathrm{supp}}

  \newcommand{\NN}{\mathbb{N}}
  \newcommand{\ZZ}{\mathbb{Z}}
  \newcommand{\RR}{\mathbb{R}}

  \newcommand{\fC}{\mathfrak{C}}
  \newcommand{\fe}{\mathfrak{e}}

  \newcommand{\ffi}{\mathfrak{i}}
  \newcommand{\fu}{\mathfrak{u}}

  \title[$g$-theorems for the Veronese construction]{Enumerative $g$-theorems for the Veronese construction for formal power series and graded algebras}

  \author{Martina Kubitzke} \thanks{Martina Kubitzke was supported by the Austrian Science Foundation (FWF) through grant Y463-N13. }
    \address{Fakult\"at f\"ur Mathematik, Universit\"at Wien, Garnisongasse 3, A-1090 Wien, Austria}
    \email{martina.kubitzke@univie.ac.at}
  \author{Volkmar Welker}
    \address{Fachbereich Mathematik und Informatik, Philipps-Universit\"at Marburg, D-35032 Marburg, Germany} 
    \email{welker@mathematik.uni-marburg.de}

\begin{document}

  \begin{abstract}
    Let $(a_n)_{n \geq 0}$ be a sequence of integers such that
    its generating series satisfies $\sum_{n \geq 0} a_nt^n = \frac{h(t)}{(1-t)^d}$ for some polynomial $h(t)$.
    For any $r \geq 1$ we study the coefficient sequence 
    of the numerator polynomial $h_0(a^{\langle r \rangle}) + \cdots + h_{\lambda'}(a^{\langle r \rangle}) t^{\lambda'}$ of the 
    $r$\textsuperscript{th} Veronese series $a^{\langle r \rangle}(t) = \sum_{n \geq 0} a_{nr} t^n$. 
    Under mild hypothesis we show that the vector of successive differences of this sequence up to the 
    $\lfloor \frac{d}{2} \rfloor$\textsuperscript{th} entry is the $f$-vector of a simplicial complex for 
    large $r$. In particular, the sequence satisfies the consequences of the unimodality part of the 
    $g$-conjecture. We give applications of the main result to Hilbert series of Veronese algebras 
    of standard graded algebras and the $f$-vectors of edgewise subdivisions of simplicial complexes.
  \end{abstract}

  \maketitle

\section{Introduction}

  We consider rational formal power series 
  \begin{eqnarray} \label{eq:powerseries}
    a(t) = \sum_{n \geq 0} a_nt^n & = & \frac{h_0(a) + \cdots + h_\lambda(a) t^\lambda}{(1-t)^d},
  \end{eqnarray}
  with integer coefficients $(a_n)_{n \geq 0}$, where $h_\lambda(a) \neq 0$. 
  We will always use the convention that $h_i(a) = 0$ if $i > \lambda$. 
  For an integer $r \geq 1$ we are interested in the 
  \emph{$r$\textsuperscript{th} Veronese series}
  \begin{eqnarray*} 
    a^{\langle r \rangle} (t) := \sum_{n \geq 0} a_{nr} t^n  & = & 
           \frac{h_0(a^{\langle r \rangle}) + \cdots + h_{\lambda'}(a^{\langle r \rangle}) t^{\lambda'}}{(1-t)^d},
  \end{eqnarray*} 

  where $h_{\lambda'}(a^{\langle r \rangle}) \neq 0$. We call $h(a) := (h_0(a) , \ldots, h_\lambda(a))$ the \emph{$h$-vector} of the 
  rational series $a(t)$ and $g(a) := (g_0(a) , \ldots, g_{\lfloor \frac{\lambda}{2} \rfloor}(a) )$ the \emph{$g$-vector} of $a(t)$, 
  where $g_0(a) := h_0(a)$ and $g_i(a) := h_i(a) - h_{i-1}(a)$ for $1 \leq i \leq \lfloor \frac{\lambda}{2} \rfloor$. If we write $a(t)$ as $b_1(t) + \frac{b_2(t)}{(1-t)^d}$
  for polynomials $b_1(t)$, $b_2(t)$ where $b_2(t)$ is of degree $< d$, then we call $\chi(a) := b_1(0)$ the \emph{characteristic} of $a(t)$.
  We are interested in enumerative properties of 
  $h(a^{\langle r \rangle})$ and $g(a^{\langle r \rangle})$ for large $r$. Our approach builds on results from 
  \cite{BW-Veronese} where the transformation $h(a) \mapsto h(a^{\langle r \rangle})$ was described as a linear 
  transformation and uses ideas from \cite{MuraiManifold} where results similar in spirit to ours were derived for 
  barycentric subdivisions of simplicial complexes. 
  Algebraic, analytic, enumerative and probabilistic aspects of the transformation $h(a) \mapsto h(a^{\langle r \rangle})$
  have recently been studied in a series of papers connecting this transformation to numerous other mathematical objects 
  (see for example \cite{Beck-Stapledon,Borodin-Diaconis-Fulman,BW-Veronese,Diaconis-Fulman,Moll-Robins-Soodhalter}).

  The following is our main result. In the formulation of the result we 
  denote by $f(\Delta):=(f_{-1}(\Delta),f_0(\Delta),\ldots,f_{d-1}(\Delta))$ the \emph{$f$-vector} of a 
  $(d-1)$-dimensional simplicial complex $\Delta$, where 
  \begin{equation*}
    f_i(\Delta) := \Big|\{F\in\Delta~|~\dim F=i\}\Big|
  \end{equation*}
  for $-1\leq i\leq d-1$.  
  Recall, that $\dim F := |F|-1$ for $F\in \Delta$ and that the \emph{dimension} $\dim \Delta$ of $\Delta$ equals $\max_{F \in \Delta} \dim F$.  
   
  \begin{theorem} \label{theorem:main}
    Let $a(t) = \sum_{n \geq 0} a_n t^n = \frac{h_0(a) + \cdots + h_\lambda(a) t^\lambda}{(1-t)^d}$ be a rational formal power series with
    integer coefficient sequence $(a_n)_{n \geq 0}$, where $h_\lambda(a) \neq 0$ and $h_0(a) = 1$. 
    \begin{itemize}
      \item[(i)] If $h_i(a) \geq 0$ for $1 \leq i \leq \lambda$ and $r\geq \max (d,\lambda)$, then there exists a simplicial complex $\Delta_r$ such that 
        \begin{equation*}
           g_i(a^{\langle r\rangle}) = f_{i-1}(\Delta_r) \qquad \qquad \mbox{ for } 0\leq i\leq \left\lfloor\frac{d}{2}\right\rfloor.
        \end{equation*}
      \item[(ii)] If there is an $N > 0$ such that $a_n > 0$ for $n > N$ and $\chi(a) \geq 0$, then there is an $R > N$ such that for 
        each $r\geq R$ and $s \geq d$, there exists a simplicial complex $\Delta_{r\cdot s}$ such that 
        \begin{equation*}
          g_i(a^{\langle r \cdot s\rangle})=f_{i-1}(\Delta_{r \cdot s}) \qquad \qquad \mbox{ for } 0\leq i\leq \left\lfloor\frac{d}{2}\right\rfloor.
        \end{equation*}
      \end{itemize}
  \end{theorem}

  Our main motivation and application for the preceding result lies in the study of Hilbert series of Veronese algebras of standard graded algebras. 
  Let $k$ be a field and $A = \bigoplus_{n\geq 0} A_n$ be a standard graded 
  $k$-algebra of dimension $d$. For $r \geq 1$ the \emph{$r$\textsuperscript{th} Veronese algebra} of $A$ 
  is the $k$-algebra $A^{\langle r \rangle } :=\bigoplus_{n\geq 0}A_{nr}$, which again 
  is standard graded of dimension $d$. The \emph{Hilbert series}
  $\Hilb(A,t) := \sum_{n \geq 0} \dim_k A_n t^n$ is a rational formal power series of the form \eqref{eq:powerseries}.   
  By definition we have $\Hilb(A^{\langle r \rangle},t) = \Hilb(A^{\langle r \rangle},t)$. 
  We call $h(A) := h(\Hilb(A,t))$ the \emph{$h$-vector} of $A$ and $g(A) := g(\Hilb(A,t))$ the \emph{$g$-vector} of $A$. 
  Using the fact that Cohen-Macaulayness of a standard graded $k$-algebra $A$ implies that $h(A) > 0$ \cite[Cor. 4.1.10]{BH-book} 
  the following corollary is a direct consequence of \ref{theorem:main}:

  \begin{corollary} \label{cor:algebra}
    Let $A$ be a $d$-dimensional standard graded $k$-algebra with Hilbert series $\Hilb(A,t) = \frac{h_0(a) + \cdots + h_\lambda(a) t^\lambda}{(1-t)^d}$.
    \begin{itemize}
      \item[(i)] If $A$ is Cohen-Macaulay and $r\geq \max(\lambda,d)$, then there exists a simplicial complex $\Delta_r$ such that 
         \begin{equation*}
           g_i(A^{\langle r\rangle})=f_{i-1}(\Delta_r) \qquad \qquad \mbox{ for } 0\leq i\leq \left\lfloor\frac{d}{2}\right\rfloor.
         \end{equation*}
      \item[(ii)] If $\chi(A)\geq 0$, then there is an $R > 0$ such that for $r\geq R$ and $s \geq d$, there exists a simplicial complex $\Delta_{r\cdot s}$ such that 
         \begin{equation*}
           g_i(A^{\langle r \cdot s \rangle})=f_{i-1}(\Delta_{r \cdot s}) \qquad \qquad \mbox{ for } 0\leq i\leq \left\lfloor\frac{d}{2}\right\rfloor.
         \end{equation*}
     \end{itemize}
  \end{corollary} 

  As a second application of our results, we can deduce properties of $h$- and $g$-vectors of the 
  edgewise subdivision of a simplicial complex. Since the definition of edgewise subdivision requires some preparation we postpone the
  formulation of these results till \ref{sec:edgewise}. In \ref{sec:further} we formulate further consequences of \ref{theorem:main} and
  some questions. \ref{sec:proof} contains the proof of \ref{theorem:main} which will be based on a detailed analysis of the transformation 
  of the $h$-vector under the Veronese transformation given in \ref{sec:analysis}.

\section{Analysis of the $h$-vector transformation}
  \label{sec:analysis}

  In \cite{BW-Veronese} it is shown that the map $h(a) \mapsto h(a^{\langle r \rangle})$ is 
  given by a linear transformation with positive integer coefficients.
  The coefficients are defined as follows. For integers $r \geq 0$, $d \geq 1$ and $i$, let
  \begin{equation*}  
    \fC(r,d,i):=\left\{(u_1,\ldots,u_d)\in\ZZ^d~\middle|~\substack{u_1+\cdots+u_d=i,\\0 \leq u_l\leq r\text{ for }1\leq l\leq d}\right\}.
  \end{equation*}
  Then set $C(r,d,i) := |\fC(r,d,i)|$ for $d\geq 1$ and $C(r,0,i):=\delta_{0,i}$.
 
  \begin{theorem} \cite[Cor. 1.2]{BW-Veronese}\label{transform2}
    Let $a(t) = \sum_{n \geq 0} a_n t^n = \frac{h_0(a)+h_1(a)t+\cdots +h_\lambda(a) t^\lambda}{(1-t)^d}$ and $h_\lambda(a) \neq 0$.
    Then for any $r\geq 1$ we have
    \begin{equation*}    
      a^{<r>}(t) =\frac{h_0(a^{<r>})+h_1(a^{<r>})t+\cdots +h_m(a^{<r>})t^m}{(1-t)^d},
    \end{equation*}
    where $m:=\max(\lambda,d)$ and 
    \begin{equation*}
      h_i^{<r>}=\sum_{j=0}^s C(r-1,d,ir-j)h_j(a)     
    \end{equation*}
    for $i=0,\ldots,m$.
  \end{theorem}
  
  In the following we focus on properties of the numbers $C(r,d,i)$. 
  As a first result, we show that they exhibit a certain symmetry and satisfy a recurrence relation.

  \begin{lemma} \label{lemma:rec}
    \begin{itemize}
      \item[(i)] For $r\geq 1$, $d\geq 0$ and $0\leq i,j\leq d$ it holds that
        \begin{equation*}
          C(r,d,i)=C(r,d,dr-i).
        \end{equation*}
      \item[(ii)] For $r\geq 1$, $d\geq 1$, $i\geq 0$ it holds that
        \begin{equation*}
          C(r,d,i)=\sum_{m=0}^{r}C(r,d-1,i-m).
        \end{equation*} 
    \end{itemize}
  \end{lemma}

  \begin{proof}
    \begin{itemize}
      \item[(i)] 
        Consider the map 
        $$
          \Phi : \Big\{
          \begin{array}{ccc}
            \fC(r,d,i)       & \rightarrow & \fC(r,d,dr-i)\\
            (u_1,\ldots,u_d) & \mapsto     & (r-u_1,\ldots,r-u_d).
          \end{array}
        $$
        This map is easily seen to be a bijection between the two given sets. Now the claim follows from the definition of $C(r,d,i)$.
      \item[(ii)] Let $(u_1,\ldots,u_d)\in\fC(r,d,i)$. 
        Then $u_1=m$ for some $0\leq m\leq r$ and $(u_2,\ldots,u_d)\in \fC(r,d-1,i-m)$.
        This implies the claimed recursion.
    \end{itemize}
  \end{proof}
 
  \smallskip 

  The following example illustrates the symmetry stated in the last lemma.

  \begin{example}\label{Ex:C}
    For $r\geq d\geq 1$ let $C^{r,d}=(C(r-1,d,ir-j))_{\substack{0\leq i\leq d\\0\leq j\leq r-1}}$ be the matrix which 
    describes the $h$-vector transformation from \ref{transform2}.  
    For $r=9$ and $d=4$, we obtain 
    \begin{eqnarray*}
      C^{9,4}=
        \begin{pmatrix} 
           1   &     0  &      0  &    0   &     0  &     0   &      0   &      0   &     0\\
         216   &    165 &     120 &    84  &     56 &     35  &      20  &      10  &     4\\                                                      
         456   &    480 &     489 &   480  &    456 &    420  &     375  &      324 &    270\\
          56   &    84  &     120 &   165  &    216 &    270  &     324  &      375 &    420 \\
           0   &     0  &      0  &    0   &     1  &     4   &      10  &      20  &     35 \\
        \end{pmatrix}.
    \end{eqnarray*}
    Note that the submatrix of $C^{9,4}$ which consists of the last four rows and columns is symmetric. 
    More generally, a similar argument as in the proof of \ref{lemma:rec} (i) shows that we have
    \begin{equation}\label{sym:ext}
      C(r-1,d,ir-j)=C(r-1,d,(d+1-i)r-(r-(j-d)))
    \end{equation}
    for $j\geq d+1$.
  \end{example}

  To simplify notation we need a few definitions.

  \begin{definition}
    Let $d\geq 1$, $r\geq 1$ and $k \geq 0$ be integers.
    \begin{itemize}
      \item[(i)] For $0 \leq k \leq r-1$ let $C^{r,d}_k \in \ZZ^{d+1}$ be the vector whose $(i+1)$\textsuperscript{st} entry equals 
         $C(r-1,d,ir-k)$ for $0\leq i\leq d$. Equivalently, $C^{r,d}_k$ is the transpose of the
         $(k+1)$\textsuperscript{st} column of $C^{r,d}$.
         For $k \geq r$ we set $C^{r,d}_k \in \ZZ^{d+1}$ to be the all zero vector. 
      \item[(ii)] For $0 \leq k \leq r-1$ let $\hat{g}^{r,d}_k=(g_0,g_1,\ldots,g_{\lfloor\frac{d}{2}\rfloor+1})\in\ZZ^{\lfloor\frac{d}{2}\rfloor+2}$ 
         be the vector defined by $g_0:=C(r-1,d,0\cdot r-k) = \delta_{0,k}$ and $g_i:=C(r-1,d,i\cdot r-k)-C(r-1,d,(i-1)\cdot r-k)$ 
         for $1\leq i\leq \lfloor\frac{d}{2}\rfloor+1$. 
         For $k \geq r$ we set $\hat{g}^{r,d}_k \in \ZZ^{\lfloor\frac{d}{2}\rfloor+2}$ to be the all zero vector. 
      \item[(iii)] 
         Let $g^{r,d}_k\in\mathbb{Z}^{\lfloor\frac{d}{2}\rfloor+1}$ be the vector obtained from $\hat{g}^{r,d}_k$ by deleting its 
         last entry.
    \end{itemize}
  \end{definition}

  For example in \ref{Ex:C} we have $C^{9,4}_1=(0,165,480,84,0)$. 
  We can rewrite the $h$-vector and induced $g$-vector transformation for power series of the form \eqref{eq:powerseries} 
  in terms of the vectors $C^{r,d}_k$ and $g^{r,d}_k$.

  \begin{remark} \label{rem:g-transform}
    Let 
    \begin{equation*}
      a(t) =\frac{h_0+h_1(a)t+\cdots +h_\lambda(a) t^\lambda}{(1-t)^d}.
    \end{equation*} 
    be a rational series with $h_\lambda(a) \neq 0$.  
    For $r\geq 1$ we have

    \begin{center}
      \begin{tabular}{ccc} 
          $\displaystyle{h(a^{\langle r\rangle})=\sum_{k=0}^\lambda h_k(a)C^{r,d}_k}$
        & 
          ~~and~~
        & 
        $\displaystyle{g(a^{\langle r\rangle})=\sum_{k=0}^\lambda h_k(a)g^{r,d}_k}.$
      \end{tabular}
    \end{center}
  \end{remark}

  For $b\in \ZZ$ and a vector $v=(v_1,\ldots,v_d)\in\ZZ^d$ we use the notation
  \begin{equation*}
    (b,v) :=(b,v_1,\ldots,v_d)\qquad \mbox{and} \qquad (v,b):=(v_1,\ldots,v_d,b).
  \end{equation*}
  Moreover, we denote by $\last(v)$ the rightmost entry of $v$.
  We record some first observations about the vectors $g^{r,d}_k$ and $\hat{g}^{r,d}_k$. 

  \begin{lemma} \label{lemma:sym1}
    Let $1\leq d\leq r$ and $0\le k\leq d$ be integers. 
    \begin{itemize}
      \item[(i)] If $d$ is odd, then $\last(\hat{g}^{r,d}_k)+\last(\hat{g}^{r,d}_{d-k})=0$.
      \item[(ii)] If $d$ is even, then $\last(\hat{g}^{r,d}_k)+\last(g^{r,d}_{d-k})=0$.
    \end{itemize}
  \end{lemma}
  \begin{proof}
    The proof follows exactly the same steps as the proof of Lemma 2.7 in \cite{MuraiManifold} except
    of the use of \ref{lemma:rec} instead of \cite[Lem. 2.5 (ii)]{MuraiManifold}.
  \end{proof}

  The next lemma provides some helpful recursions for the vectors $C^{r,d}_k$, $g^{r,d}_k$ and $\hat{g}^{r,d}_k$.

  \begin{lemma}
    Let $1\leq d\leq r$ and $0\le k\leq d$ be integers. Then:
    \begin{itemize}
      \item[(i)] 
        \begin{equation*}
          C^{r,d}_k = \sum_{j=0}^{k-1}(0,C^{r,d-1}_j)+\sum_{j=k}^{r-1}(C^{r,d-1}_{j},0).
        \end{equation*}
      \item[(ii)]
        If $d$ is even, then 
        \begin{equation}\label{eq:even1}
          g^{r,d}_k=\sum_{l=0}^{k-1}(0,g^{r,d-1}_l)+\sum_{l=k}^{r-1}\hat{g}^{r,d-1}_l.
        \end{equation}
        In particular, 
        \begin{equation}
          \last(g^{r,d}_k)=\sum_{l=0}^{k-1}\last(g^{r,d-1}_l)+\sum_{l=k}^{r-1}\last(\hat{g}^{r,d-1}_l).
        \end{equation}
      \item[(iii)] If $d$ is odd, then
        \begin{equation}\label{eq:odd1}
          \hat{g}^{r,d}_k=\sum_{l=0}^{k-1}(0,g^{r,d-1}_l)+\sum_{l=k}^{r-1}\hat{g}^{r,d-1}_l.
        \end{equation}
        In particular,
        \begin{equation}\label{eq:odd2}
          \last(\hat{g}^{r,d}_k)=\sum_{l=0}^{k-1}\last(g^{r,d-1}_l)+\sum_{l=k}^{r-1}\last(\hat{g}^{r,d-1}_l).
        \end{equation}
    \end{itemize}
  \end{lemma}

  \begin{proof}
   (i) is a rephrasing of the recursion given in \ref{lemma:rec} (ii). Part (ii) and (iii) follow by a straightforward computation from 
   the definition of the vectors $g^{r,d}_k$ and $\hat{g}^{r,d}_k$ and \ref{lemma:rec} (ii).
  \end{proof}

  We will further need the following technical lemma, which analyzes the behavior of the last entries of the vectors $\hat{g}^{r,d}_k$ 
  for $k\geq d+1$ (as opposed to \ref{lemma:sym1} where this is achieved for $0\leq k\leq d$).

  \begin{lemma}\label{lemma:symG}
    Let $1\leq d\leq r$ and $d+1\leq k\leq r-1$ be integers.
    \begin{itemize}
      \item[(i)] If $d$ is even, then $\last(\hat{g}^{r,d}_k)+\last(\hat{g}^{r,d}_{r-(k-d)} )=0$.
      \item[(ii)] If $d$ is odd, then
        {\small
        \begin{eqnarray*} 
            \last(\hat{g}^{r,d}_k)+C(r-1,d,\frac{d+1}{2}r+k-d) & = & C(r-1,d, \frac{d-1}{2}r+k-d)
        \end{eqnarray*}
        }
    \end{itemize}
  \end{lemma}

  \begin{proof}
    \begin{itemize}
       \item[(i)] 
         Let $d$ be even and $d+1\leq k\leq r-1$. By definition it holds that
         {\tiny 
         \begin{eqnarray*} 
           \last(\hat{g}^{r,d}_k) & = & C\left(r-1,d,(\frac{d}{2}+1)r-k\right)-C\left(r-1,d,\frac{d}{2}r-k\right)\\
                                  & = & C\left(r-1,d,\frac{d}{2}r-(r-(k-d))\right)-C\left(r-1,d,(\frac{d}{2}+1)r-(r-(k-d))\right)\\
                                  & = & -\last(\hat{g}^{r,d}_{r-(k-d)}),
         \end{eqnarray*}
         }
         where the second equality follows from the symmetry in \eqref{sym:ext}. 
         This shows (i).
       \item[(ii)] 
         Let $d$ be odd and $d+1\leq k\leq r-1$.
         Since $d$ is odd, it holds that $\lfloor\frac{d}{2}\rfloor=\frac{d-1}{2}$ and thus by definition
         {\tiny
         \begin{eqnarray*}
           \last(\hat{g}^{r,d}_k) & = & C\left(r-1,d,(\frac{d-1}{2}+1)r-k\right)-C\left(r-1,d,\frac{d-1}{2}r-k\right)\\
                                  & \stackrel{(\ref{sym:ext})}{=} & C\left(r-1,d,\frac{d+1}{2}r-(r-(k-d))\right)-C\left(r-1,d,\frac{d+3}{2}r-(r-(k-d))\right).
         \end{eqnarray*}
         }
         This finishes the proof of (ii).
    \end{itemize}
  \end{proof}

  \smallskip

  The following example provides a list of the vectors $C^{r,d}_k$, $g^{r,d}_k$ and $\hat{g}^{r,d}_k$ 
  for small values of $d$.

  \begin{example}\label{Ex:SmallValues}
    \begin{center}
      \begin{tabular}{|l|l|l|l|l|}
        \hline
        $\empty$ & $\empty$ & $C^{r,d}_k$ & $g^{r,d}_k$ & $\hat{g}^{r,d}_k$ \\
        \hline
        $d=1$    & $k=0$ & $(1,0)$ & $(1)$ & $(1,-1)$\\
        $\empty$ & $k=1$ & $(0,1)$ & $(0)$ & $(0,1)$\\ 
        $\empty$ & $k=2,\ldots, r-1$ & $(0,1)$ & $(0)$ & $(0,1)$\\
        \hline
        $d=2$    & $k=0$ &  $(1,r-1,0)$ & $(1,r-2)$ & $(1,r-2,-r+1)$\\
        $\empty$ & $k=1$ & $(0,r,0)$ & $(0,r)$   & $(0,r,-r)$\\
        $\empty$ & $k=2,\ldots, r-1$ & $(0,r-k+1,k-1)$ & $(0,r-k+1)$  & $(0,r-k+1,2k-r-2)$\\
        \hline
        $d=3$    & $k=0$ & $(1,7,1,0)$ & $(1,6)$ & $(1,6,-6)$ \\
        $r=3$    & $k=1$ & $(0,6,3,0)$ & $(0,6)$ & $(0,6,-3)$\\
        $\empty$ & $k=2$ & $(0,3,6,0)$ & $(0,3)$ & $(0,3,3)$ \\
        \hline
        $d=4$    & $k=0$ & $(1,31,31,1,0)$ & $(1,30,0)$ & $(1,30,0,-30)$ \\
        $r=4$    & $k=1$ & $(0,20,40,4,0)$ & $(0,20,20)$ & $(0,20,20,-36)$\\
        $\empty$ & $k=2$ & $(0,10,44,10,0)$ & $(0,10,34)$ & $(0,10,34,-34)$ \\
        $\empty$ & $k=3$ & $(0,4,40,20,0)$ & $(0,4,36)$ & $(0,4,36,-20)$\\
        \hline
      \end{tabular}
    \end{center}
  \end{example}

  Note that in the above example, the entries of the vectors $g^{r,d}_k$ are exclusively non-negative 
  whereas the last entry of $\hat{g}^{r,d}_k$ can also be negative. Our next aim is to prove that this is true in general.

  \begin{lemma}\label{lem:nonnegative}
    Let $1\leq d\leq r$ and $0\leq k\leq r-1$ be integers. Then:
    \begin{itemize}
      \item[(i)] $g^{r,d}_k$ is non-negative and $\last(g^{r,d}_k)=0$ if and only if $d$ is even, $d=r$ and $k=0$, or $d=1$ and $k > 0$. 
      \item[(ii)] If $d$ is odd, then $\last(\hat{g}^{r,d}_k)>0$ for $k\geq \frac{d}{2}$ and $\last(\hat{g}^{r,d}_k)<0$ for $k< \frac{d}{2}$.
      \item[(iii)] If $d$ is even, then $\last(\hat{g}^{r,d}_k)> 0$ for $k> \frac{d+r}{2}$ and $\last(\hat{g}^{r,d}_k)< 0$ for $k< \frac{d+r}{2}$.
        If $r$ is even, then $\last(\hat{g}^{r,d}_{\frac{d+r}{2}})=0$.
    \end{itemize}
  \end{lemma}

  \begin{proof}
    For $d\leq 2$ the statements follow from \ref{Ex:SmallValues}.
    Let $d>2$. 

    \medskip

    \noindent {\sf Case:} $d$ even.

    \begin{itemize}
      \item[(i)] 
        By Equation \eqref{eq:even1} we can express $g^{r,d}_k$ as
        \begin{equation*}
          g^{r,d}_k=\sum_{l=k}^{r-1}\hat{g}^{r,d-1}_l+\sum_{l=0}^{k-1}(0,g^{r,d-1}_l).
        \end{equation*}
        From part (i) of the induction hypothesis we infer that all but the last entry of $g^{r,d}_k$ are non-negative. 
        The last entry is given as
        \begin{equation}\label{eq:last}
          \last(g^{r,d}_k)=\sum_{l=k}^{r-1}\last(\hat{g}^{r,d-1}_l)+\sum_{l=0}^{k-1}\last(g^{r,d-1}_l).
        \end{equation}

        We now distinguish the cases $k \leq \frac{d}{2}$ and $k< \frac{d}{2}$.

        \begin{itemize}
          \item[$\triangleright$] $k\geq \frac{d}{2}$. The induction hypothesis (ii) implies that $\last(\hat{g}^{r,d-1}_k)>0$. Since the first sum on the 
            right-hand side of \eqref{eq:last} contains at least one summand and since $\last(g^{r,d-1}_l)\geq 0$ by the induction hypothesis 
           (i) we conclude that $\last(g^{r,d}_k)>0$.
          \item[$\triangleright$] $k< \frac{d}{2}$. \ref{lemma:sym1} (i) implies that $\sum_{l=k}^{d-1-k}\last(\hat{g}^{r,d-1}_k)=0$. We therefore obtain

            \begin{equation*}
              \label{eq:lastlast}
              \last(g^{r,d}_k)=\sum_{l=d-k}^{r-1}\last(\hat{g}^{r,d-1}_l)+\sum_{l=0}^{k-1}\last(g^{r,d-1}_l).
            \end{equation*}

            Since $d-1 > 1$ and $l\geq \frac{d}{2}$ for $d-k\leq l$ and $k< \frac{d}{2}$, all summands in the first and the second sum on the right-hand 
            side of Equation \eqref{eq:last} are strictly positive by the induction hypothesis (ii) and (i), respectively. Hence, $\last(g^{r,d}_k)\geq 0$. 
            Furthermore, $\last(g^{r,d}_k)=0$ if and only if $d-k>r-1$ and $k=0$, equivalently $d>r-1$ and $k=0$, equivalently $d=r$ and $k=0$. 
        \end{itemize}

    \item[(iii)] 
       Let $k>\frac{d+r}{2}$.
       Using \ref{lemma:rec} (ii) and the definition of $\hat{g}^{r,d}_k$ we obtain that the last entry of $\hat{g}^{r,d}_k$ can be written as
       \begin{eqnarray*}
         \last(\hat{g}^{r,d}_k) & = & \sum_{l=0}^{k-1}\last(\hat{g}^{r,d-1}_l) + \\
                              &   & \sum_{l=k}^{r-1}\left(C(r-1,d-1,\frac{d+2}{2}r-l)-C(r-1,d-1,\frac{d}{2}r-l)\right)\\
                              & = & \sum_{l=0}^{d-1}\last(\hat{g}^{r,d-1}_l)+\sum_{l=d}^{d+r-k-1}\last(\hat{g}^{r,d-1}_l)+\sum_{l=d+r-k}^{k-1}\last(\hat{g}^{r,d-1}_l) +\\
                              &   & \sum_{l=k}^{r-1}\left(C(r-1,d-1,\frac{d+2}{2}r-l)-C(r-1,d-1,\frac{d}{2}r-l)\right).
      \end{eqnarray*} 
      By \ref{lemma:sym1} (i) it holds that $\sum_{l=0}^{d-1}\last(\hat{g}^{r,d-1}_l)=0$. \ref{lemma:symG} (ii) further implies
      \begin{equation*}
        \sum_{l=k}^{r-1}\Big(C(r-1,d-1,\frac{d+2}{2}r-l)-C(r-1,d-1,\frac{d}{2}r-l)\Big)=-\sum_{l=d}^{d+r-k-1}\last(\hat{g}^{r,d-1}_l).
      \end{equation*}
      The expression for $\last(\hat{g}^{r,d}_k)$ thus simplifies to
      \begin{equation*}
        \last(\hat{g}^{r,d}_k)=\sum_{l=d+r-k}^{k-1}\last(\hat{g}^{r,d-1}_l).
      \end{equation*}
      From the induction hypothesis (ii) we finally conclude that $\last(\hat{g}^{r,d}_k) \geq 0$.
      Furthermore, we have $\last(\hat{g}^{r,d}_k)=0$ if and only if $k-1<d+r-k$, i.e.,
      $k<\frac{d+r+1}{2}$. Since we also have $k\geq\frac{d+r}{2}$ the last condition is true
      if and only if $k=\frac{d+r}{2}$ and $r$ is even. 
      From the symmetry in \ref{lemma:symG} (i) it further follows that $\last(\hat{g}^{r,d}_k)<0$ for 
      $d+1\leq k<\frac{d+r}{2}$. Since, by \ref{lemma:sym1} (ii), $\last(\hat{g}^{r,d}_k)=-\last(g^{r,d}_{d-k})$ 
      for $0\leq k\leq d$ we deduce from (i) that $\last(\hat{g}^{r,d}_k)<0$ for $0\leq k\leq d$. 
    \end{itemize}

    \noindent{\sf Case:} $d$ odd.

    \begin{itemize}
      \item[(i)] Using \eqref{eq:odd1} we can write the vector $\hat{g}^{r,d}_k$ as 
        \begin{equation*}
          \hat{g}^{r,d}_k = \sum_{l=k}^{r-1}\hat{g}^{r,d-1}_l+\sum_{l=0}^{k-1}(0,g^{r,d-1}_l).
        \end{equation*}
        It follows from part (i) of the induction hypothesis that $g^{r,d}_k$ is non-negative. Since
        the first sum on the right-hand side of the last equation contains the term $\hat{g}^{r,d-1}_{r-1}$
        which by the induction hypothesis (iii) has a strictly positive last entry, we see that
        $\last(g^{r,d}_k)>0$. 
      \item[(ii)] 
        By \eqref{eq:odd2} we have
        \begin{equation}\label{eq:proof}
          \last(\hat{g}^{r,d}_k)=\sum_{l=k}^{r-1}\last(\hat{g}^{r,d-1}_l)+\sum_{l=0}^{k-1}\last(g^{r,d-1}_l).
        \end{equation}  
        Assume first that $k<d$. By \ref{lemma:symG} (i) it holds that $\sum_{l=d}^{r-1}\last(\hat{g}^{r,d-1}_l)=0$. 
        This yields
        \begin{eqnarray*} 
          \last(\hat{g}^{r,d}_k) & = & \sum_{l=k}^{d-1}\last(\hat{g}^{r,d-1}_l)+\sum_{l=0}^{d-1-k}\last(g^{r,d-1}_l)+\sum_{l=d-k}^{k-1}\last(g^{r,d-1}_l)\\
                                 & = & \sum_{l=d-k}^{k-1}\last(g^{r,d-1}_l),
        \end{eqnarray*}
        where the last equality holds by \ref{lemma:sym1} (ii). For $k\geq \frac{d}{2}$ the last sum consists
        at least of the summand $\last(g^{r,d-1}_{\frac{d-1}{2}})$ and the induction hypothesis (i) finally implies $\last(\hat{g}^{r,d}_k)>0$.
        This combined with \ref{lemma:sym1} (i) also shows that $\last(\hat{g}^{r,d}_k)<0$ for $0\leq k\leq \frac{d}{2}$.

        Now let $k\geq d$. The induction hypothesis (i) implies that $\sum_{l=0}^{k-1}\last(g^{r,d-1}_l)>0$. If $k>\frac{d-1+r}{2}$, then we infer from part (iii) of  
        the induction hypothesis that $\sum_{l=k}^{r-1}\last(\hat{g}^{r,d-1}_l)>0$ and the claim follows from \eqref{eq:proof}. 
        Assume that $k\leq \frac{d-1+r}{2}$. In this case we can write
        \begin{equation*}
          \sum_{l=k}^{r-1}\last(\hat{g}^{r,d-1}_l)=\sum_{l=k}^{r-k+d}\last(\hat{g}^{r,d-1}_l)+\sum_{l=r-k+d+1}^{r-1}\last(\hat{g}^{r,d-1}_l).
        \end{equation*}

        By \ref{lemma:symG} (i) it holds that $\sum_{l=k}^{r-k+d}\last(\hat{g}^{r,d-1}_l)=0$ and from the induction hypothesis (iii) we know
        that $\sum_{l=r-k+d+1}^{r-1}\last(\hat{g}^{r,d-1}_l)\geq 0$. The above reasoning together with \eqref{eq:proof} finally completes the proof.
    \end{itemize}
  \end{proof}

  The next lemma is concerned with the behavior of the vectors $g^{r,d}_k$ when $r$ grows.

  \begin{proposition}\label{prop:growth}
    Let $1\leq d\leq r$ and $0\leq k\leq r$ be integers. Then:
    \begin{itemize}
      \item[(i)] $g^{r,d}_k\leq g^{r+1,d}_k$ (componentwise).
      \item[(ii)] If $d$ is odd, then $\last(\hat{g}^{r+1,d}_k)\geq \last(\hat{g}^{r,d}_k)$ for $k\geq \frac{d}{2}$.
 \item[(iii)] If $d$ is even, then $\last(\hat{g}^{r+1,d}_{k+1})\geq \last(\hat{g}^{r,d}_k)$ for $k\geq d$.
    \end{itemize}
  \end{proposition}

  \begin{proof}
    We proceed by induction on $d$ and show all three parts of the lemma simultaneously. 
    For $d\leq 2$, \ref{Ex:SmallValues} verifies the assertion. Now let $d>2$.  

    \medskip

   \noindent {\sf Case:} $d$ even. 

    \begin{itemize}
      \item[(i)] 
 				By Equation \eqref{eq:even1} the vector $g^{r+1,d}_k$ can be computed in the following way:
    		\begin{equation*}
      		g^{r+1,d}_k = \sum_{l=0}^{k-1}(0,g^{r+1,d-1}_l)+\sum_{l=k}^{r}\hat{g}^{r+1,d-1}_l.
    		\end{equation*}
    		In the sequel, let $\bar{v}$ denote the vector which is obtained from a vector $v$ by deleting its last entry. 
    		It follows from the induction hypothesis (i) that
    		\begin{align*}
      		\bar{g}^{r+1,d}_k & \geq \sum_{l=0}^{k-1}\overline{(0,g^{r,d-1}_l)}+\sum_{l=k}^{r-1} g^{r,d-1}_l+g^{r+1,d-1}_r\\
                        		& =\bar{g}^{r,d}_k + g^{r+1,d-1}_r \geq \bar{g}^{r,d}_k,
    		\end{align*}
    		where the last inequality holds by \ref{lem:nonnegative} (i). 
    		It remains to show the desired inequality for the last entries of $g^{r+1,d}_k$ and $g^{r,d}_k$. First assume that 	  					$k\geq \frac{d}{2}$. 
    		In this case, it holds that
    		\begin{align*}
      		\last(g^{r+1,d}_k) & =\sum_{l=0}^{k-1}\last(g^{r+1,d-1}_l)+\sum_{l=k}^{r}\last(\hat{g}^{r+1,d-1}_l)\\
                        		 & \geq \sum_{l=0}^{k-1}\last(g^{r,d-1}_l) + \sum_{l=k}^{r-1}\last(\hat{g}^{r+1,d-1}_l) +  	\last(\hat{g}^{r+1,d-1}_r)\\
                         		 & \geq \sum_{l=0}^{k-1}\last(g^{r,d-1}_l) + \sum_{l=k}^{r-1}\last(\hat{g}^{r,d-1}_l) + \last(\hat{g}^{r+1,d-1}_r)\\
                         		 & \geq \sum_{l=0}^{k-1}\last(g^{r,d-1}_l) + \sum_{l=k}^{r-1}\last(\hat{g}^{r,d-1}_l)
    	\end{align*}
    	Here the first and the second inequality follow from the induction hypothesis (i) and (ii), respectively. The third							inequality is a consequence of \ref{lem:nonnegative} (ii). 
    	Next, let $k<\frac{d}{2}$. In this case, we infer from \ref{lemma:sym1} (i) that   								$\sum_{l=k}^{d-1-k}\last(\hat{g}^{r+1,d-1}_l)=0$. Combining 
    	this with the previous reasoning we can therefore conclude that
    	\begin{align*}
      	\last(g^{r+1,d}_k) & \geq \sum_{l=0}^{k-1}\last(g^{r,d-1}_l) + 		      \sum_{l=d-k}^{r-1}\last(\hat{g}^{r+1,d-1}_l)+\last(\hat{g}^{r+1,d-1}_r)\\
                         	 & \geq \sum_{l=0}^{k-1}\last(g^{r,d-1}_l) + \sum_{l=d-k}^{r-1}\last(\hat{g}^{r+1,d-1}_l),
    	\end{align*}
    	where the last inequality follows from \ref{lem:nonnegative} (ii). 
    	For $l\geq d-k$ and $k<\frac{d}{2}$ it holds that $l>\frac{d}{2}$. Hence, part (ii) of the induction hypothesis (ii) implies 
    	$\last(\hat{g}^{r+1,d-1}_l)\geq\last(\hat{g}^{r,d-1}_l)$ for $d-k\leq l\leq r-1$. We finally obtain 
    	\begin{align*}
      	\last(g^{r+1,d}_k) & \geq \sum_{l=0}^{k-1}\last(g^{r,d-1}_l)+\sum_{l=d-k}^{r-1}\last(\hat{g}^{r,d-1}_l)\\                     
                         & =\sum_{l=0}^{k-1}\last(g^{r,d-1}_l)+\sum_{l=k}^{r-1}\last(\hat{g}^{r,d-1}_l)\\
                         & =\last(g^{r,d-1}_k).
        \end{align*}
        For the second equality we use that by \ref{lemma:sym1} (i) we have  $\sum_{l=k}^{d-k-1}\last(\hat{g}^{r,d-1}_l)=0$. 

      \item[(iii)] We have shown in the proof of \ref{lem:nonnegative} (iii) that for $k\geq d$ the last entry of $\hat{g}^{r,d}_k$ can be expressed as 
	\begin{equation}\label{eq:lastEntry}
		\last(\hat{g}^{r,d}_k)=\sum_{l=d+r-k}^{k-1}\last(\hat{g}_l^{r,d-1}). 
	\end{equation}
	Using part (ii) of the induction hypothesis and \ref{lem:nonnegative} (ii) we deduce
	\begin{align*}
	\last(\hat{g}^{r,d}_k)&\leq \sum_{l=d+r-k}^{k-1}\last(\hat{g}_l^{r+1,d-1})\\
	&\leq \sum_{l=d+r-k}^{k-1}\last(\hat{g}_l^{r+1,d-1})+\last(\hat{g}^{r+1,d-1}_k)\\
	&=\last(\hat{g}^{r+1,d}_{k+1}),
	\end{align*}
	where the last equality holds by \eqref{eq:lastEntry}.
    \end{itemize}

    \noindent {\sf Case:} $d$ odd. 
    \begin{itemize}
       \item[(i)] 
         It follows from \eqref{eq:odd1} that the vector $g^{r+1,d}_k$ can be computed as
         \begin{align*}
            g^{r+1,d}_k & =\sum_{l=0}^{k-1}\overline{(0,g^{r+1,d-1}_l)}+\sum_{l=k}^{r} g^{r+1,d-1}_l \\
                        & \geq \sum_{l=0}^{k-1}\overline{(0,g^{r,d-1}_l)} + \sum_{l=k}^{r-1}g^{r,d-1}_l + g^{r+1,d-1}_r \\
                        & \stackrel{\eqref{eq:odd1}}{=} g^{r,d-1}_k + g^{r+1,d-1}_ r \\
                        & \geq g^{r,d}_k.
         \end{align*}
         Here, the first and the last inequality follow from the induction hypothesis (i) and \ref{lem:nonnegative} (i), respectively. 

       \item[(ii)] 
         First, suppose that $\frac{d}{2}\leq k<d$. 
         We have shown in the proof of \ref{lem:nonnegative} (ii) that the last entry of $\hat{g}^{r,d}_k$ can be written as
         \begin{equation*}
           \last(\hat{g}^{r,d}_k)=\sum_{l=d-k}^{k-1} \last(g^{r,d-1}_l).
         \end{equation*}
         Using this expression and part (i) of the induction hypothesis we obtain
         \begin{equation*}
           \last(\hat{g}^{r,d}_k)\leq\sum_{l=d-k}^{k-1} \last(g^{r+1,d-1}_l)=\last(\hat{g}^{r+1,d}_k).
         \end{equation*}
  
         Next, assume that $k\geq d$. The last entry of $\hat{g}^{r+1,d}_k$ is given by
         \begin{equation*}
           \last(\hat{g}^{r+1,d}_k)=\sum_{l=0}^{k-1}\last(g^{r+1,d-1}_l)+\sum_{l=k}^{r}\last(\hat{g}^{r+1,d-1}_l).
         \end{equation*}
         Moreover, \ref{lemma:symG} (i) implies that $\sum_{l=k}^{r+1-k+d}\last(\hat{g}^{r+1,d-1}_l)=0$. 
         Using this and part (i) of the induction hypothesis we conclude
         \begin{align*}
           \last(\hat{g}^{r+1,d}_k)&=\sum_{l=0}^{k-1}\last(g^{r+1,d-1}_l)+\sum_{l=r+2-k+d}^{r}\last(\hat{g}^{r+1,d-1}_l)\\
           &\geq \sum_{l=0}^{k-1}\last(g^{r,d-1}_l)+\sum_{l=r+2-k+d}^{r}\last(\hat{g}^{r+1,d-1}_l).
         \end{align*} 
         Applying part (iii) of the induction hypothesis to the last sum, it follows that
         \begin{align*}
           \last(\hat{g}^{r+1,d}_k)&\geq \sum_{l=0}^{k-1}\last(g^{r,d-1}_l)+\sum_{l=r+2-k+d}^{r}\last(\hat{g}^{r,d-1}_{l-1})\\
  	     &=\sum_{l=0}^{k-1}\last(g^{r,d-1}_l)+\sum_{l=k}^{r-k+d}\last(\hat{g}^{r,d-1}_l)+\sum_{l=r+1-k+d}^{r-1}\last(\hat{g}^{r,d-1}_{l})\\
  	     &=\sum_{l=0}^{k-1}\last(g^{r,d-1}_l)+\sum_{l=k}^{r-1}\last(\hat{g}^{r,d-1}_l)\\
  	     &=\last(\hat{g}^{r,d}_k).
         \end{align*}
         For the first equality we have used that $\sum_{l=k}^{r-k+d}\last(\hat{g}^{r,d-1}_l)=0$ by \ref{lemma:symG} (i).
    \end{itemize}
  \end{proof}


  As a consequence of the last lemma and of \ref{lem:nonnegative} we obtain the following relation between the vectors 
  $g^{1,d}_k$ and $g^{r,d}_k$ for $r\geq d$.

  \begin{corollary}\label{cor:g-growth}
    Let $1\leq d\leq r$ and $0\leq k\leq d-1$ be integers. Then:
    \begin{equation*}
      g^{1,d}_k \leq g^{d,d}_k \leq \cdots \leq g^{r,d}_k \leq g^{r+1,d}_k.
    \end{equation*}
  \end{corollary}

  \begin{proof}
    All inequalities except for the leftmost follow from \ref{prop:growth}.
    It is easily seen that 
    \begin{equation*}
      g^{1,d}_k =
      \begin{cases}
         (0,\;\;\,0,0, \ldots, 0) \;\;\; \quad\mbox{ if } k \neq 0 \\
         (1,-1,0,\ldots, 0)       \;\;\; \quad\mbox{ if } k=0 \\
      \end{cases}
    \end{equation*}
    The claim now follows directly from \ref{lem:nonnegative} (i).
  \end{proof}

\section{Proof of \ref{theorem:main}}\label{Section:main}
  \label{sec:proof}

  We now recall the definition and important properties of admissible vectors introduced by Murai in \cite{MuraiManifold}. 

  \begin{definition}
    Let $f=(1,f_0,f_1,\ldots, f_d)\in \ZZ^{d+2}$ be the $f$-vector of a simplicial complex (We do not assume $f_d\neq 0$) and 
    let $\alpha=(0,1,\alpha_0,\ldots,\alpha_{d-1})$, $\beta\in \ZZ^{d+2}$.
    \begin{itemize}
      \item[(i)] The vector $\alpha$ is called a \emph{basic admissible vector} for $f$ if $(1,\alpha_0,\ldots,\alpha_{d-1})$ is the $f$-vector of a
        simplicial complex and $f_i\geq \alpha_i$ for $0\leq i\leq d-1$.
      \item[(ii)] The vector $\beta$ is called \emph{admissible} for $f$ if there exist vectors $\beta^{(1)},\ldots,\beta^{(t)}\in \ZZ^{d+2}$ such that 
        $\beta= \beta^{(1)}+\cdots +\beta^{(t)}$ and $\beta^{(k)}$ is a basic admissible vector for $f+\beta^{(1)}+\cdots +\beta^{(k-1)}$ 
        for $1\leq k\leq t$.
    \end{itemize}
  \end{definition}
  The following lemma states some properties of admissible vectors which will be crucial for the proof of \ref{theorem:main} (i).
  \begin{lemma}\cite[Lem. 3.1]{MuraiManifold} \label{lem:admissible}
    Let $f=(1,f_0,f_1,\ldots,f_d)\in\ZZ^{d+2}$ be the $f$-vector of a simplicial complex and let 
    $\alpha=(0,\alpha_{-1},\alpha_0,\ldots,\alpha_{d-1})\in \ZZ^{d+2}$
    be admissible for $f$. Then:
    \begin{itemize}
      \item[(i)] $f+\alpha$ is the $f$-vector of a simplicial complex.
      \item[(ii)] If $g\in\ZZ^{d+2}$ is the $f$-vector of a simplicial complex and $g\geq f$ (componentwise), then $\alpha$ is admissible for $g$.
      \item[(iii)] If $\beta\in\ZZ^{d+2}$ is admissible for $f$, then $\alpha +\beta$ is admissible for $f$.
      \item[(iv)] For any integer $0\leq b\leq \alpha_{d-1}$, the vector $(0,1,\alpha_0,\ldots,\alpha_{d-2},\alpha_{d-1}-b)$ is admissible for $f$.
    \end{itemize}
  \end{lemma}

  Now \ref{theorem:main} (i) follows from the next proposition, \ref{lem:admissible} and \ref{rem:g-transform}.   
  Note that \ref{pr:admissible} and \ref{lem:admissible} require $r \geq d$. If $r \geq \lambda$ then 
  \ref{pr:admissible} (ii) implies that $g^{r,d}_k$ is admissible for $g^{r,d}_0$ for $1\leq k\leq \lambda$.
  Hence by \ref{rem:g-transform} and \ref{lem:admissible} the assertion follows for $r \geq \max(d,\lambda)$.
 
%
  \begin{proposition}
    \label{pr:admissible}
    Let $r\geq d>0$ be positive integers. Then:
    \begin{itemize}
      \item[(i)] $g^{r,d}_0$ is the $f$-vector of a simplicial complex.
      \item[(ii)] $g^{r,d}_k$ is admissible for $g^{r,d}_0$ for $1\leq k\leq r$.
      \item[(iii)] If $d$ is odd and $k\geq \frac{d}{2}$, then $\hat{g}^{r,d}_k$ is admissible for $(g^{r,d}_0,0)$.
    \end{itemize}
  \end{proposition}

  Note, that \ref{theorem:main} (i) already follows from \ref{pr:admissible} (i) and (ii). However, we will use assertion (iii) to show (i) and (ii) by induction.

  \begin{proof}
    Throughout the proof we denote for a vector $v$ by $\bar{v}$ the vector obtained from $v$ by deleting its last entry. 

    We proceed by induction on $d$ and prove (i), (ii) and (iii) simultaneously. 
    For $d\leq 2$ the claim follows from \ref{Ex:SmallValues}.  Now let $d>2$. 

    \medskip

    \noindent {\sf Case:} $d$ odd. 

    \begin{itemize}
      \item[(i)] 
        We obtain from \eqref{eq:odd1}
        \begin{equation} \label{eq:ProofOdd}
          g^{r,d}_0 = \sum_{l=0}^{r-1}g^{r,d-1}_l.
        \end{equation}

        It follows from the induction hypothesis (i) that $g^{r,d-1}_0$ is the $f$-vector of a simplicial complex and
        from (ii) it follows that $g^{r,d-1}_l$ is admissible for $g^{r,d-1}_0$ for $1\leq l\leq r-1$. \ref{lem:admissible} (iii) implies that
        $\sum_{l=1}^{r-1} g^{r,d-1}_l$ is admissible for $g^{r,d-1}_0$. From part (i) of the same lemma and \eqref{eq:ProofOdd} we infer that 
        $g^{r,d}_0$ is the $f$-vector of a simplicial complex. This shows (i).
    
      \item[(ii)] 
        From \eqref{eq:odd1} we deduce that

        \begin{equation*}
          g^{r,d}_k = \sum_{l=k}^{r-1} g^{r,d-1}_l + \sum_{l=0}^{k-1}\overline{(0,g^{r,d-1}_l)}.
        \end{equation*}

        By part (ii) of the induction hypothesis $g^{r,d-1}_l$ is admissible for $g^{r,d-1}_0$ for $k\leq l\leq r-1$. Thus, from \ref{lem:admissible} 
        we infer that $\sum_{l=k}^{r-1} g^{r,d-1}_l$ is admissible for $g^{r,d-1}_0$.
        \ref{lem:nonnegative} (i) combined with \eqref{eq:ProofOdd} yields that $g^{r,d}_0 \geq g^{r,d-1}_0$. Since we have already shown that
        $g^{r,d}_0$ is the $f$-vector of a simplicial complex it now follows from \ref{lem:admissible} (ii) that $\sum_{l=k}^{r-1} g^{r,d-1}_l$ is
        also admissible for $g^{r,d}_0$. It remains to prove that $\sum_{l=0}^{k-1}\overline{(0,g^{r,d-1}_l)}$ is admissible for $g^{r,d}_0$.
        The claim then follows from \ref{lem:admissible} (iii).
        We know from the induction hypothesis and \ref{lem:admissible} (i) that $\sum_{l=0}^{k-1}g^{r,d-1}_l$ is the $f$-vector of a simplicial complex.
        Removal of the last entry preserves this property, thus $\sum_{l=0}^{k-1}\bar{g}^{r,d-1}_l$ is the $f$-vector of a simplicial complex. Since 
        by \eqref{eq:ProofOdd} and \ref{lem:nonnegative} it further holds that $\sum_{l=0}^{k-1} \bar{g}^{r,d-1}_l\leq \bar{g}^{r,d}_0$ we conclude 
        that $(0,\sum_{l=0}^{k-1}\bar{g}^{r,d-1}_l)$ is a basic admissible vector for $\bar{g}^{r,d}_0$ which finishes the proof of (ii).

      \item[(iii)] 
        By \eqref{eq:odd1} we have that
        \begin{align*}
          \hat{g}^{r,d}_k & =\sum_{l=k}^{r-1}\hat{g}^{r,d-1}_l+\sum_{l=0}^{k-1}(0,g^{r,d-1}_l)\\
                          & =\sum_{l=k}^{d-1}\hat{g}^{r,d-1}_l+\sum_{l=d}^{r-1}\hat{g}^{r,d-1}_l+\sum_{l=0}^{k-1}(0,g^{r,d-1}_l)
        \end{align*}

        Since by \ref{lemma:symG} (i) $\sum_{l=d}^{r-1}\last(\hat{g}^{r,d-1}_l)=0$ it follows that
        $\sum_{l=d}^{r-1}\hat{g}^{r,d-1}_l = \sum_{l=d}^{r-1}(g^{r,d-1}_l,0)$. \ref{lem:nonnegative} states
        that $\last(\hat{g}^{r,d-1}_l)<0$ for $k\leq l\leq d-1$, i.e., 
        $\hat{g}^{r,d-1}_l \leq (g^{r,d-1}_l,0)$ (componentwise). We thus obtain

        \begin{align*}
          \hat{g}^{r,d}_k & \leq \sum_{l=k}^{d-1}(g^{r,d-1}_l,0)+\sum_{l=d}^{r-1}(g^{r,d-1}_l,0)+\sum_{l=0}^{k-1}(0,g^{r,d-1}_l)\\
                          & =\sum_{l=k}^{r-1}(g^{r,d-1}_l,0)+\sum_{l=0}^{k-1}(0,g^{r,d-1}_l).
        \end{align*}

        The proof of (ii) shows that the right-hand side of the above inequality is admissible for $(g^{r,d}_0,0)$. Since
        by \ref{lem:nonnegative} (i) and (ii) $\hat{g}^{r,d}_k$ is non-negative and since the inequality above is an equality
        for all but the last entry it follows from \ref{lem:admissible} (iv) that $\hat{g}^{r,d}_k$ is admissible for $(g^{r,d}_0,0)$.
      \end{itemize}

    \noindent {\sf Case:} $d$ even. 

    \begin{itemize}
      \item[(i)] 
        Equation \eqref{eq:even1} implies

        \begin{align} 
           g^{r,d}_0 & =\sum_{l=0}^{r-1}\hat{g}^{r,d-1}_l \notag \\
                     & =\sum_{l=0}^{d-1}(g^{r,d-1}_l,\last(\hat{g}^{r,d-1}_l))+\sum_{l=d}^{r-1}\hat{g}^{r,d-1}_l. \label{eq:ProofEven}
        \end{align}

        By \ref{lemma:sym1} (i) it holds that $\sum_{l=0}^{d-1}\last(\hat{g}^{r,d-1}_l)=0$. By the induction hypothesis (i) and (ii)
        and \ref{lem:admissible} $\sum_{l=1}^{d-1} g^{r,d-1}_l$ is admissible for $g^{r,d-1}_0$. Hence,
        $\sum_{l=1}^{d-1} (g^{r,d-1}_l,0)$ is admissible for $(g^{r,d-1}_0,0)$. By part (iii) of the induction hypothesis $\hat{g}^{r,d-1}_l$
        is admissible for $(g^{r,d-1}_0,0)$ as well. \ref{lem:admissible} (i) together with the above reasoning implies that
        $g^{r,d}_0$ is the $f$-vector of a simplicial complex. 

      \item[(ii)]  
        Equation \eqref{eq:odd1} yields

        \begin{align*}
          \hat{g}^{r,d}_k & =\sum_{l=k}^{r-1} \hat{g}^{r,d-1}_l + \sum_{l=0}^{k-1}(0,g^{r,d-1}_l)\\
                          & =\sum_{l=0}^{k-1}(0,g^{r,d-1}_l) + \sum_{l=k}^{\frac{d}{2}-1}\hat{g}^{r,d-1}_l + \sum_{l=\frac{d}{2}}^{r-1}\hat{g}^{r,d-1}_l \\
                         & \leq \sum_{l=0}^{k-1}(0,g^{r,d-1}_l) + \sum_{l=k}^{\frac{d}{2}-1}(g^{r,d-1}_l,0)+\sum_{l=\frac{d}{2}}^{r-1}\hat{g}^{r,d-1}_l,
        \end{align*}

        where the last inequality follows from \ref{lem:nonnegative} (ii).
        Since this inequality is an equality for all entries but the last one it follows from \ref{lem:admissible} (iv) that it suffices to show that
        the right-hand side of the above inequality is admissible for $g^{r,d}_0$.
        By the induction hypothesis $g^{r,d-1}_l$ is admissible for $g^{r,d-1}_0$ which by \ref{lem:admissible} (i) and (iii) implies that
        $\sum_{l=0}^{k-1} g^{r,d-1}_l$ is the $f$-vector of a simplicial complex. From \eqref{eq:ProofEven} and \ref{lem:nonnegative} (i)  we deduce that 
        $\sum_{l=0}^{k-1} g^{r,d-1}_l \leq \bar{g}^{r,d}_0$. This finally shows that $\sum_{l=0}^{k-1}(0,g^{r,d-1}_l)$ is admissible for $g^{r,d}_0$.
        Furthermore, by the induction hypothesis, $g^{r,d-1}_l$ is admissible for $g^{r,d-1}_0$ for $k\leq l\leq \frac{d}{2}-1$ and thus, by \ref{lem:admissible} (iii)
        $\sum_{l=k}^{\frac{d}{2}-1} g^{r,d-1}_l$ is admissible for $g^{r,d-1}_0$. In particular, $\sum_{l=k}^{\frac{d}{2}-1}(g^{r,d-1}_l,0)$ is admissible for $(g^{r,d-1}_0,0)$.
        Since by \ref{lem:nonnegative} (i) and \eqref{eq:ProofEven} it holds that $g^{r,d}_0\geq (g^{r,d-1}_0,0)$ \ref{lem:admissible} (ii) implies
        that $\sum_{l=k}^{\frac{d}{2}-1}(g^{r,d-1}_l,0)$ is admissible for $g^{r,d}_0$.
        By (iii) of the induction hypothesis $\hat{g}^{r,d-1}_l$ is admissible for $(g^{r,d-1}_0,0)$ for $\frac{d}{2}\leq l\leq r-1$. 
        This finishes the proof of (ii).
    \end{itemize}
  \end{proof}

  It remains to verify \ref{theorem:main} (ii).

  \begin{proof}[Proof of \ref{theorem:main} (ii)]
    The assumptions imply that after passing to any Veronese hi\-gher than the $(N+1)$\textsuperscript{st} 
    all coefficients of the power series are non-negative. In addition for sufficiently high Veronese the
    series will be of the form $\chi(A) + \frac{b_2(t)}{(1-t)^d}$ for a polynomial
    $b_2(t)$ of degree $< d$. 
    By Theorem 1.2 from \cite{Beck-Stapledon} it then follows that there is an $R > N$ such that for $r\geq R$ 
    the series $\left( \frac{b_2(t)}{(1-t)^d}\right)^{\langle r \rangle}$ has all coefficients of its numerator polynomial 
    positive and is for degree $\leq d$. Theorem 1.1 \cite{BW-Veronese} then implies that the coefficients of the numerator
    polynomial go to infinity when taking higher Veronese series, except for the constant coefficient. Hence by 
    $\chi(A) \geq 0$ high enough Veronese series of $a(t)$ will have a numerator polynomial with positive 
    coefficients. Now the assertion follows from \ref{theorem:main} (i). 
  \end{proof}

\section{Further results, questions and applications}

  \subsection{Further results and questions} We first prove a monotonicity result on the entries of the $g$-vectors. 
  \label{sec:further}

  \begin{proposition}
    Let $a(t) = \sum_{n \geq 0} a_nt^n  =  \frac{h_0(a) + \cdots + h_\lambda(a) t^\lambda}{(1-t)^d}$ with $h_\lambda(a) \neq 0$.
    If $h_i(a)\geq 0$ for $0\leq i\leq \lambda$, then 
    \begin{equation*}
      g(a) \leq g(a^{\langle d\rangle})\leq g(a^{\langle d+1\rangle})\leq g(a^{\langle d+2\rangle}) \leq \cdots \qquad . 
    \end{equation*}
  \end{proposition}

  \begin{proof}
    The claim follows directly from the $g$-vector transformation stated in \ref{rem:g-transform} combined with 
    \ref{cor:g-growth}.
  \end{proof}

  Recall that a sequence $(a_0,\ldots,a_t)\in\mathbb{N}^{t+1}$ is called an \emph{$M$-sequence} if it is the 
  Hilbert function of a $0$-dimensional standard graded Artinian $k$-algebra. Macaulay 
  gave a characterization of such sequences by means of 
  numerical conditions (see e.g., \cite[Thm. 4.2.10]{BH-book}). 
  In particular, it is well-known that $f$-vectors of simplicial complexes are $M$-sequences. As a consequence of \ref{cor:algebra} we thus obtain.

  \begin{remark}
   \label{rem:algebra}
    Let $A$ be a $d$-dimensional standard graded $k$-algebra.
    \begin{itemize}
      \item[(i)] If $A$ is Cohen-Macaulay and $r\geq d$, then $g(A^{\langle r\rangle})$ is an $M$-sequence.
      \item[(ii)] If $\chi(A)\geq 0$, then there is an $R > 0$ such that for $r\geq R$ and $s \geq d$, the vector $g(A^{\langle r \cdot s \rangle})$ is an $M$-sequence.
    \end{itemize}
  \end{remark}
 
  Satoshi Murai has explained to us an algebraic proof of part (i) of the preceding remark. It is based on a suitable linear system of
  parameters and a linear form that multiplies injectively up to the middle degree.
  
  In the case $A$ is Cohen-Macaulay the $h$-vector of $A$ and all its Veronese algebras is an $M$-sequence. The latter follows since 
  Veronese algebras of Cohen-Macaulay algebras are again Cohen-Macaulay by a result from \cite{GotoWatanabe}.
  For general standard graded algebras $A$ with $\chi(A) \geq 0$ the numerator polynomial of the Hilbert series of high
  Veronese algebras of $A$ will have positive coefficients but already here it is not clear if it finally will become an
  $M$-sequence. Indeed the following is true.

  \begin{proposition}
    Let $a(t) = \sum_{n \geq 0} a_n t^n = \frac{h_0(a) + \cdots + h_\lambda(a) t^\lambda}{(1-t)^d}$ be a rational formal power series with
    integer coefficient sequence $(a_n)_{n \geq 0}$, where $h_\lambda(a) \neq 0$ and $h_0(a) = 1$. 
    If there is an $N > 0$ such that $a_n > 0$ for $n > N$ and $\chi(a) \geq 0$, then there is an $R > N$ such that for 
    each $r\geq R$ the coefficient sequence of the numerator polynomial of $a^{\langle r \rangle}(t)$ is an $M$-sequence.
  \end{proposition}
  \begin{proof}
    As in the proof of \ref{theorem:main} we deduce that the numerator polynomial of $a^{\langle r \rangle}(t)$ will have
    positive coefficients for large enough $r$. Now by \cite[Thm. 1.4]{BW-Veronese} this polynomial will also be
    real rooted for large enough $r$. Then the result follows from \cite[Thm. 3.6]{Bell-Skandera} where it is shown that
    the coefficient sequence of a real rooted polynomial $1+c_1t+\cdots + c_dt^d$ with positive integer
    coefficients is an $M$-sequence.
  \end{proof}

  The preceding proposition raises the question if real rootedness of a polynomial $1+c_1t+\cdots + c_dt^d$ with positive integer
  coefficients implies more than the coefficient sequence being an $M$-sequence. Indeed, already in \cite{Bell-Skandera}  
  Bell and Skandera conjecture that the assumptions imply that the coefficient sequence is the $f$-vector of a simplicial
  complex. Here we would like to ask the question if indeed real rootedness already implies the consequences of 
  \ref{theorem:main} (i).

  \begin{question}
    Let $1+c_1t+\cdots + c_dt^d$ be a real rooted polynomial with positive integer coefficients. Assume that 
    $1 < c_1 < \cdots < c_l$. Is $(1,c_1-1,c_2-c_1,\ldots, c_l-c_{l-1})$ a $f$-vector of a simplicial complex?
  \end{question}
 
  Another interesting question arises from \ref{theorem:main} (ii). First the limiting behavior does not give a single bound
  $R$ such that for $r > R$ the assertions are valid for $r$\textsuperscript{th} Veronese series, rather the bound
  depends on a `starting' parameter. Second it is not clear to us if the assumption $\chi(A) \geq 0$ is really needed.
  Indeed, under the remaining assumptions high Veronese series will have a numerator polynomial with positive coefficients
  except for the highest coefficient which is $\chi(A)$. These two observations motivate the following questions.

  \begin{question}
    Assume there is an $N > 0$ such that $a_n > 0$ for $n > N$.
    Does there exist an $R > N$ such that for
    each $r\geq R$, there exists a simplicial complex $\Delta_{r}$ such that
    \begin{equation*}
      g_i(a^{\langle r \rangle})=f_{i-1}(\Delta_{r }) \qquad \qquad \mbox{ for } 0\leq i\leq \left\lfloor\frac{d}{2}\right\rfloor? 
    \end{equation*}
  \end{question}

  The main result of \cite{NevoPetersenTenner} suggests another possible strengthening of \ref{theorem:main}. Indeed in \cite{NevoPetersenTenner} 
  the authors conclude that the $\gamma$-vector of the barycentric subdivision of homology sphere is the $f$-vector of a balanced simplicial 
  complex. But the conclusion of \ref{theorem:main} (i) cannot be modified in this direction, indeed the $g$-vector of the $9$\textsuperscript{th} 
  of the polynomial ring in $8$ variables is not the $f$-vector of a balanced simplicial complex.

  \subsection{Application to edgewise subdivisions}
  \label{sec:edgewise}
  In this section we review the edgewise subdivision of a simplicial complex $\Delta$ (see \cite{Edelsbrunner}) 
  and its relation to the Veronese algebras of the Stanley-Reisner ring of $\Delta$ (see \cite{BrunR}).
  This allows us to apply the results obtained in \ref{Section:main}
  to the $g$-vectors of those complexes.

  Before we proceed to edgewise subdivision we recall some basic definitions. Let $k$ be a field. 
  For an abstract simplicial complex $\Delta$ on vertex set 
  $[n] := \{1, \ldots, n\}$ the \emph{Stanley-Reisner ring}
  $k[\Delta]$ of $\Delta$ is the quotient of the polynomial ring $k[x_1, \ldots, x_n]$ by the
  Stanley-Reisner ideal $I_\Delta := \langle \prod_{i \in F} x_i~|~F \not\in \Delta\rangle$ generated
  by the squarefree monomials whose support does not lie in $\Delta$. 
  The \emph{$h$-vector} $h(\Delta)=h(k[\Delta])$ of $\Delta$ is defined as the $h$-vector of $k[\Delta]$.
  We call $g(\Delta) := g(k[\Delta])$ the \emph{$g$-vector} of $\Delta$.

  We are now ready to describe the construction of the $r$\textsuperscript{th} edgewise subdivision of a simplicial 
  complex. Let $\Delta$ be a simplicial complex on vertex set $[n]$ and 
  let $r\geq 1$ be a positive integer. Set $\Omega_r:=\{(i_1,\ldots,i_n)\in \NN^n~|~i_1+\cdots+i_n=r\}$. 
  Denote by $\fe_i$ the $i$\textsuperscript{th} unit vector of $\RR^n$. By the obvious 
  identification, we can consider $\Delta$ as a simplicial complex over the 
  vertex set $\Omega_1=\{\fe_1,\ldots,\fe_n\}$. 
  For $i\in [n]$ set $\fu_i:=\fe_i+\cdots+\fe_n$ and 
  for $a=(a_1,\ldots,a_n)\in\ZZ^n$, $a_l \geq 0$ for $1 \leq l \leq n$, let 
  $\ffi(a):=\sum_{l=1}^n a_l\cdot \fu_l$. The \emph{$r$\textsuperscript{th} edgewise 
  subdivision} of $\Delta$ is the simplicial complex $\Delta(r)$ on ground 
  set $\Omega_r$ such that $F \subseteq \Omega_r$ is a simplex in $\Delta(r)$ 
  if and only if 

  \begin{enumerate}
    \item[(i)]$\bigcup_{a\in F}\supp(a)\in\Delta$
    \item[(ii)] For all $a, \tilde{a}\in F$ either 
        $\ffi(a-\tilde{a})\in\{0,1\}^n$ or 
        $\ffi(\tilde{a}-a)\in\{0,1\}^n$.
  \end{enumerate}
 
  The following result by Brun and R\"omer \cite{BrunR} links the $r$\textsuperscript{th} edgewise 
  subdivision of a simplicial complexe to the $r$\textsuperscript{th} Veronese algebra 
  of its Stanley-Reisner ring. The formulation of the result requires some familiarity with the
  basic theory of Gr\"obner bases; see for example \cite{Adams}.

  \begin{proposition} \label{prop:initial}
    Let $\Delta$ be a simplicial complex on ground set $[n]$ and let $r \geq 1$. 
    Set $S(r) := k[y_{i_1,\ldots, i_n}~|~(i_1,\ldots, i_n) \in \Omega_r ]$ 
    and let $I(r)$
    be such that $k[\Delta]^{<r>}=S(r)/I(r)$.
    Then there is a term order $\preceq$ for which $I_{\Delta(r)}$ is the 
    initial ideal of $I(r)$.
  \end{proposition}

  By basic facts on initial ideals and Hilbert functions it follows that
  \begin{equation*}
    \Hilb(S(r)/I(r),t)=\Hilb(k[\Delta(r)],t).
  \end{equation*}
  Thus, the $r$\textsuperscript{th} Veronese algebra of the Stanley-Reisner ring of a simplicial 
  complex $\Delta$ and the $r$\textsuperscript{th} edgewise subdivision of this complex have
  the same $h$- and $g$-vector, respectively. From this, we infer that the $h$- and $g$-vectors of
  edgewise subdivisions of simplicial complexes satisfy the same conditions as the 
  $h$- and $g$-vectors of Veronese algebras of Stanley-Reisner rings. Also recall that a simplicial complex $\Delta$ 
  is called \emph{Cohen-Macaulay over a field $k$} if $k[\Delta]$ is a Cohen-Macaulay ring. Therefore, by the fact that
  the numerator polynomial of $\Hilb(k[\Delta],t)$ has degree $\leq d$ the \ref{cor:algebra}
  immediately implies the following.
 
  \begin{corollary}
    Let $\Delta$ be a $(d-1)$-dimensional simplicial complex 
    and let $\Delta(r)$ be the $r$\textsuperscript{th} edgewise subdivision of $\Delta$.  
    \begin{itemize}
      \item[(i)] If $\Delta$ is Cohen-Macaulay over some field and $r\geq d$, then there exists a simplicial complex $\Delta_r$ such that 
         \begin{equation*}
           g_i(\Delta^{\langle r\rangle})=f_{i-1}(\Delta_r) \qquad \qquad \mbox{ for } 0\leq i\leq \left\lfloor\frac{d}{2}\right\rfloor.
         \end{equation*}
      \item[(ii)] Then there exists an $R>0$ such that for $r\geq R$ and $s\geq d$, there exists a simplicial complex $\Gamma_{r\cdot s}$ such that 
        \begin{equation*}
          g_i(\Delta(r\cdot s))=f_{i-1}(\Gamma_{r\cdot s}) \qquad \qquad \mbox{ for } 0\leq i\leq \left\lfloor\frac{d}{2}\right\rfloor.
        \end{equation*}
        In particular, $g(\Delta(r\cdot s))$ is an $M$-sequence.
    \end{itemize}
  \end{corollary} 

  \section{Acknowledgment}
  
    We thank Satoshi Murai for pointing out to us a mistake in the formulation 
    of \ref{theorem:main} (i) in a preliminary version of the paper and 
    explaining to us an algebraic proof of \ref{rem:algebra} (i). 
    We are grateful to Eran Nevo for helpful comments and suggestions.
    
  \bibliography{biblio}
  \bibliographystyle{plain}
\end{document}